\documentclass[11pt]{article}

\usepackage{amsmath, amsthm, amsfonts, amssymb}
\usepackage{color}

\hoffset= - 0.5 cm \voffset=  0.1 cm \numberwithin{equation}{section}

\newtheorem{theorem}{Theorem}[section]
\newtheorem{lemma}[theorem]{Lemma}
\newtheorem{proposition}[theorem]{Proposition}
\newtheorem{corollary}[theorem]{Corollary}
\newtheorem{remark}[theorem]{Remark}
\newtheorem{definition}[theorem]{Definition}
\newtheorem{example}[theorem]{Example}
\newtheorem{hypothesis}[theorem]{Hypothesis}

\newcommand{\bth}{\begin{theorem}}
\renewcommand{\eth}{\end{theorem}}
\newcommand{\bpr}{\begin{proposition}}
\newcommand{\epr}{\end{proposition}}
\newcommand{\bco}{\begin{corollary}}
\newcommand{\eco}{\end{corollary}}
\newcommand{\ble}{\begin{lemma}}
\newcommand{\ele}{\end{lemma}}

\newcommand{\bpf}{\begin{proof}}
\newcommand{\epf}{\end{proof}}
\newcommand{\bex}{\begin{example}}
\newcommand{\eex}{\end{example}}
\newcommand{\bdf}{\begin{definition}}
\newcommand{\edf}{\end{definition}}
\newcommand{\bre}{\begin{remark}}
\newcommand{\ere}{\end{remark}}

\newcommand{\beq}{\begin{equation}}
\newcommand{\eeq}{\end{equation}}
\newcommand{\bal}{\begin{aligned}}
\newcommand{\eal}{\end{aligned}}
\newcommand{\beqr}{\begin{eqnarray*}}
\newcommand{\eeqr}{\end{eqnarray*}}

\def\P{{\mathbb P}}
\def\R{{\mathbb R}}

\def\E{{\mathbb E }}
\def\N{\mathbb N}

\def\hh{{\vskip 2mm \noindent }}
 \def\vv{\vskip 1mm \noindent}

\begin{document}

\title {\Huge On
linear  evolution equations
 with  cylindrical L\'evy noise
 }

\date{}

\maketitle


\begin{center}

 Enrico Priola \footnote
 { \noindent \! \! \!
  Supported by the M.I.U.R. research projects Prin 2004 and
 2006 ``Kolmogorov equations'' and
 by the Polish Ministry of Science and Education project 1PO 3A 034
29 ``Stochastic evolution equations with L\'evy noise''.}

\vspace{ 2 mm } {\small  \it Dipartimento di Matematica,
Universit\`a di Torino, \par
  via Carlo Alberto 10,  \ 10123, \ Torino, Italy. \par
 e-mail \ \ enrico.priola@unito.it }
\\ \ \\
Jerzy Zabczyk   \footnote { \noindent \! \! Supported by the
Polish Ministry of Science and Education project 1PO 3A 034 29
``Stochastic evolution equations with \! \! L\'evy noise''.}

\vspace{ 2 mm} {\small  \it Instytut Matematyczny,   Polskiej
Akademii Nauk,
\par ul. Sniadeckich 8, 00-950, \ Warszawa, Poland.\par
 e-mail \ \ zabczyk@impan.gov.pl }
\end{center}


 \par \ \par

\vspace{1.5 mm}

\noindent {\bf Abstract:} We study an infinite-dimensional
 Ornstein-Uhlenbeck  process $(X_t)$ in a given
  Hilbert space $H$. This is
   driven by a cylindrical  symmetric  L\'evy process without a
Gaussian component and taking values in a Hilbert space $U$ which
 usually contains $H$.
We give if and only if conditions under which $X_t$ takes values
in $H$ for some $t>0$ or for all $t>0$.
  Moreover, we prove irreducibility for
 $(X_t)$.

\section{Introduction and notation}

There is an increasing interest in   stochastic evolution
equations driven by  L\`evy noise. We refer to the recent
monograph \cite{PeZ} which also discusses several
 applications.

In this note  we concentrate on  the   linear stochastic
 differential equation
 \beq
\label{ab}
\begin{cases} dX_t  = A X_t dt  +
    dZ_t, \,\,\,\,\, t \ge 0,\\
       X_0  = x \in  H,
 \end{cases}\eeq
   in a real separable Hilbert space  $H$ driven by
    an infinite dimensional  cylindrical
    symmetric  L\'evy process $Z= (Z_t)$.
 The  process  $Z$
   may  take values in a Hilbert space $U$ usually
greater than $H$.
    Moreover we assume that $A$
   is a linear possibly unbounded operator
  which generates a $C_0$-semigroup $(e^{tA})$ on
   $H$.

Solutions of \eqref{ab}, called (generalised) Ornstein-Uhlenbeck
processes,
    have recently received
  a lot of attention  (see, for instance,
  \cite{Ch},
   \cite{BRS},
    \cite{FR},
    \cite{Dawson},
   \cite{LR}, \cite{PZ4},
   \cite{PeZ}, \cite{PZ7} and \cite{BZ}). Transition semigroups determined
   by  solutions $X = (X_t^x)$ to \eqref{ab} are also
  studied under the name of
 generalized Mehler semigroups.

 In the  case  when $Z$ is a cylindrical Wiener process the theory of
 equations \eqref{ab} is well understood (see \cite{DZ}, \cite{DZ1}
  and the references therein).
  The situation changes completely
  in  the L\`evy noise case and new phenomena appear.
  For instance,  the c\`adl\`ag property  of trajectories in $H$
 can be proved only
  in   very special cases (see Remark \ref{kotelenez})  and
 is an open question in  general. Note that
  in \cite{FR} it is proved that trajectories of $(X_t^x)$
    are
  c\`adl\`ag only in some enlarged Hilbert space  containing
  $H$.

  In this note  we consider
    cylindrical L\`evy process  $Z= (Z_t)$
   defined  by
the orthogonal expansion
\begin{equation}\label{ee}
 Z_t = \sum_{n \ge 1} \beta_n Z_t^n e_n,\;\;\;\; t \ge 0,
\end{equation}
 where  $(e_n)$ is an orthonormal basis of $H$. We also   assume
\begin{hypothesis} \label {bas}
 $Z^n=(Z_t^n)$ are  independent,  real
 valued,  symmetric, identically distributed  L\'evy
  processes without a Gaussian part
   defined on a fixed stochastic basis.
   Moreover,
    $(\beta_n)$ is a given (possibly unbounded)
   sequence of  {\it positive }  real numbers.
\end{hypothesis}
In our previous paper \cite{PZ7},  we have considered the case in
which
 $(Z_t^n)$ are independent,  real
 valued, normalized, symmetric  $\alpha$-stable
  processes, $\alpha \in (0,2)$.
   For the linear equation \eqref{ab} in \cite{PZ7} we have proved
 $p$-integrability  of   trajectories in $H$, $p \in (0, \alpha)$,
   and characterized the  support
   of $(X_t^x, X_T^x)$ in $L^p(0,T; H) \times H$. Moreover, we have
 established the
   strong Feller property for the transition Markov semigroup
    associated to  \eqref{ab}.
 We are not able to prove such results in the present more general
situation.

This note can be considered as a preliminary step towards
 an extension of \cite{PZ7}  to general L\'evy processes.
 In fact
  in Theorem \ref{bo1} we provide if and only if
conditions under
 which $(X^x_t)$ takes values in $H$. It turns out that
  if there exists a positive time $t_0$ such that $X_{t_0}^x \in H$,
  $\P$-a.s.,
   then for all $t>0$, we have that $X_t^x \in H$,  $\P$-a.s..
 In Proposition \ref{bo11} we
  consider a
    class of symmetric one dimensional L\'evy processes
  $Z_t^n$, which includes the $\alpha$-stable processes,
   and which satisfies
  the conditions
  of Theorem \ref{bo1}.
  For
such processes we also show existence and uniqueness of  invariant
measure. The Markov  property and irreducibility are  proved in
  Theorems \ref{bo1} and  \ref{irr}.

The results of the paper
 apply in particular  to stochastic heat
 equations with Dirichlet boundary conditions (see Example \ref{E1}).

 Let us mention that
  in the recent paper \cite{BZ} a different cylindrical L\'evy noise
$Z$ is studied
 by subordinating a  cylindrical Wiener process, given by
  \eqref{ee} with $(Z_t^n)$ independent real valued Wiener
  processes.  It is difficult to judge at the moment which class
   of cylindrical L\'evy noises will suit  better
 modelling purposes.

 As far as the strong Feller property for \eqref{ab} is
concerned
  we stress  two different difficulties.
   One difficulty is related to the fact that very rarely   for
  non-Gaussian  infinitely divisible
 measures  in Hilbert spaces formulae for the Radon-Nikodym
 derivatives are known.
 Another problem is that the
  well-known Bismut-Elworthy-Li formula
 is not available in the non-Gaussian
 case.
 A related formula, but requiring a non trivial Gaussian
 component in the L\'evy noise,
 was established in finite dimensions in \cite{PZ6} and
 generalized to infinite dimensions in \cite{MPR}.

\ \vskip 2mm   The space  $H$ will denote  a real separable
   Hilbert space with inner product
    $\langle \cdot, \cdot \rangle$ and norm
$|\cdot|$.
 We will fix an  orthonormal basis
   $(e_n)$ in $H$. Through
    the basis $(e_n)$ we will often identify $H$
 with  $l^2$.
  More generally, for a given sequence $\rho= (\rho_n)$
 of real numbers, we set
   \beq \label{rho}
 l^2_{\rho} = \{ (x_n) \in \R^{\N} \,\, :\, \sum_{n \ge 1} x_n^2
 \rho_n^2 < + \infty \}.
\eeq
 The space
  $l^2_{\rho}$ becomes a separable Hilbert space with the inner
 product: $\langle x,y\rangle = \sum_{n \ge 1} x_n y_n
 \, \rho_n^2$, for $x = (x_n)$, $y = (y_n) \in l^2_{\rho}$.

{\vskip 2mm} Let us recall that  a  L\'evy process $(Z_t)$ with
values in $H$ is an
 $H$-valued process defined  on
 some stochastic basis
  $(\Omega, {\cal F}, ({\cal F}_t)_{t \ge 0}, \P)$,
 having stationary independent increments,
  c\`adl\`ag  trajectories,
  and such that $Z_0 =0$, $\P$-a.s..
One has that
 \begin{equation}\label{zt}
\E [e^{i \langle Z_t , s \rangle }]  = \exp ( - t \psi (s) ), \; s
\in H,
 \end{equation}
where the exponent  $\psi$ can be expressed by the following
infinite dimensional  L\'evy-Khintchine formula,
 \begin{equation}\label{psi} \psi(s)= \frac{1}{2} \langle
Q s,s \rangle - i \langle a, s\rangle - \int_{H} \Big(  e^{i \langle
s,y \rangle }  - 1 - \frac { i \langle s,y \rangle} {1 + |y|^2 }
\Big ) \nu (dy), \;\; s \in H.
 \end{equation}
 Here $Q$ is a  symmetric non-negative trace class operator on
$H$, $a \in H$ and $\nu$ is the  L\'evy measure or the jump
intensity measure  associated to $(Z_t)$, i.e., $\nu$
 is
  a $\sigma$-finite Borel measure on $H $
 such that $\nu (\{ 0\})=0$ and
   $\int_{H} (|y|^2 \wedge 1) \nu (dy) < + \infty$
 (see \cite{sato} and \cite{PeZ}).

\hh  According to Proposition  \ref{d}
  our cylindrical L\'evy process
 $Z$ appearing in \eqref{ab}
  is a L\'evy process
  taking values in the Hilbert space
   $U =l^2_{\rho}$, with a properly
 chosen weight ${\rho}$ (see  Remark \ref{levy}).

\section{The  main result  }

 Concerning equation \eqref{ab}, we make
the following assumption.
\begin{hypothesis} \label {basic}

 \noindent   $A : D(A) \subset H \to H$
 is a self-adjoint operator such that the fixed  basis $(e_n)$ of $H$
 verifies:
  $(e_n ) \subset D(A)$,  $A e_n  =  - \gamma_n e_n$
  with $\gamma_n >0$, for any  $n \ge 1$, and $\gamma_n \to +
  \infty $.
\end{hypothesis}
\noindent Clearly, under (i),
 $
 D(A) = \{  x= (x_n) \in H \, :\, \sum_{n \ge 1} x_n^2
 \gamma_n^2 < +\infty
 \}.
$ In addition  $A$ generates a compact $C_0$-semigroup
  $(e^{tA})$ on $H$ such that
$$
e^{tA} e_k = e^{- \gamma_k t} e_k,\;\; \;\; k \in \N, \;\;\; t \ge
0.
$$
Hypothesis \ref{basic} is also considered  in \cite{PZ7} when
 $(Z^n_t)$ are symmetric $\alpha$-stable L\'evy processes,
 $\alpha \in (0,2)$.

 Recall that  we are assuming that   $(Z_t^n)$
 are defined on the same
stochastic basis $(\Omega, {\cal F}, ({\cal F}_t), \P)$
 satisfying the usual assumptions.

 Since the law of  $Z_t^n$ is symmetric, we have,
 for any $n \ge 1$, $t \ge 0,$
 \begin{equation} \label{ps}
 \E [e^{i h Z^n_t}] = e^{- t \psi (h)},\;\;\; h \in \R,
\end{equation}
 where
 \beq
\label{psu}
 \psi (h)  = \int_{\R} \big( 1 -  \cos  (hy ) \big ) \nu (dy), \;\; h
 \in \R,
\eeq
 and  the L\'evy measure $\nu$ is  symmetric
  (i.e., $\nu (A)= \nu (-A)$, for any Borel set $A \subset \R$).
  This follows by the next elementary result.

  \begin{proposition} \label{use} A one dimensional L\'evy process
  $L = (L_t)$ without Gaussian part has symmetric distribution
   at some time $t>0$  if
  and only if  its L\'evy measure $\nu$ is symmetric.
   Moreover, if  $\nu$ is symmetric
 then $L_t$  has symmetric distribution
   at any time $t\ge 0$.
 \end{proposition}
 \begin{proof}
 Since $(L_t)$ has no Gaussian part, according to the
   L\'evy-Khintchine formula, we have $\E [e^{i h L_t}]
   = e^{- t \psi (h)},\; h \in \R,$ $t \ge 0$, with
   $$
  \psi (h) =  - i  a h
    - \int_{\R} \Big( e^{i  hy  } - 1 - \frac { i h y } {1
+ y^2 } \Big ) \nu (dy), \;\; h \in \R,
$$
 for some $a \in \R$.
 Define  the reflection measure $\tilde \nu$ of $\nu$, i.e.,
 $ \tilde \nu (A) = \nu (-A) $, for any Borel set $A \subset \R$.
 It is easy to check that also $\tilde \nu $ is a L\'evy measure
 and moreover
$$
\psi (-h) =  i  a h - \int_{\R} \Big( e^{i  hy  } - 1 - \frac { i
h y } {1 + y^2 } \Big ) \tilde \nu (dy), \;\;\; h \in \R.
$$
Since $L_t$ has symmetric distribution, we must have
 $\psi (h) = \psi (-h)$, $h \in \R$. By  uniqueness of the
 L\'evy-Khintchine formula (see \cite[Theorem 8.1]{sato})
 we obtain that $(-a, \tilde \nu) = (a, \nu)$. It follows that
  $a=0$ and $\nu = \tilde \nu$. Therefore $\nu$ is symmmetric.
 \end{proof}

 \medskip We need the following lemma.

\ble \label{indip}
 Let us consider a sequence of
independent, symmetric, infinitely divisible real random variables
 $\xi_n$ defined on the same probability space
  $(\Omega, {\cal F}, \P)$
  such that
$$
 \E[e^{ih \xi_n}] =
 \exp \Big[ - \int_{\R} \big( 1 -  \cos  (hy ) \big ) \nu_n (dy)
 \Big], \;\; h
 \in \R, \; n \ge 1,
$$
where $\nu_n$ are  L\'evy measures.
 The following assertions are equivalent.

\hh (i) \;\; $\sum_{n \ge 1} \xi_n^2 < + \infty, \;\;\;\;
\P-a.s.;$ \hh (ii) \;\;  $\sum_{n \ge 1} \int_{\R} \big( 1 \wedge
y^2 \big ) \nu_n (dy) < + \infty.$
 \ele
\bpf  We will use the
   following  theorem (see, for instance \cite{K},
   page 70-71): { let $U_n$
   be a sequence of independent and symmetric real random
   variables; then the following statements are equivalent:
    $\sum_{n \ge 1} U_n$ converges in
   distribution; $\sum_{n \ge 1} U_n$
  converges $\P$-a.s.; $\sum_{n \ge 1} U_n^2 $ converges
  $\P$-a.s..}

\vv  By the previous result,  assertion  (i) is equivalent to
convergence in distribution
 of the sequence of random variables $(\sum_{n = 1}^N \xi_n)$.

We have, for any $N \in \N$, $h \in \R$, using  independence,
  $$
  \E  [e^{i h \sum_{n=1}^N  \xi_n  } ] = \prod_{n=1}^N
\E  [e^{i  h   \xi_n } ] = \prod_{n=1}^N e^{-
 \int_{\R} \big( 1 -  \cos  (h  y ) \big ) \nu_n (dy)}
$$
$$
=  e^{-   \sum_{n=1}^N
 \int_{\R} \big( 1 -  \cos  (h  y ) \big ) \nu_n (dy)}
=  e^{-
 \int_{\R} \big( 1 -  \cos  (h y ) \big )
 \big (\sum_{n=1}^N \nu_n \big) (dy)}
$$
(i) $\Rightarrow$ (ii). We are assuming  convergence in
distribution
 of the sequence  $(\sum_{n = 1}^N \xi_n )$. By \cite[Theorem 8.7]
{sato} the limiting distribution $\mu$ is again symmetric and
  infinitely
 divisible;  the characteristic function  of $\mu$ is
  given by
 $$
 \exp \Big( -\frac1 2 q h^2
 -\int_{\R} \big( 1 -  \cos   (h y)  \big ) \tilde \nu (dy)
 \Big ),\;\;\;\; h \in \R,
 $$
where
\begin{equation}\label{e1}
q\geq 0 \,\,\,{\rm and}\,\,\, \int_{\R} \big( y^2 \wedge 1 )
\tilde \nu (dy) <  + \infty .
\end{equation}
Moreover, for arbitrary bounded continuous functions
 $f$ from $\R$ into
$\R$, vanishing on a neighborhood of $0$, we have \beq
\label{incr}
 \lim_{N \to \infty} \int_{\R} f(y)\,
 \big (\sum_{n=1}^N \nu_n \big) (dy)
 = \int_{\R}f(y) \tilde \nu (dy).
\eeq
 If $f$ in \eqref{incr} is in addition a non-negative
 function, then $\int_{\R} f(y)\,
 \big (\sum_{n=1}^N \nu_n \big) (dy)$ is an increasing sequence.
 Thus, for any $N \in \N$, $f: \R \to \R_+$ continuous, bounded and
vanishing on a neighborhood of $0$,
  $$
\int_{\R} f(y)\,
 \big (\sum_{n=1}^N \nu_n \big) (dy)
 \le  \int_{\R}f(y) \tilde \nu (dy).
$$
Since the function $y \mapsto y^2 \wedge 1$ is a pointwise limit
of a monotone increasing sequence of non-negative functions $f_k$
which are continuous, bounded and vanishing on a neighborhood of
$0$, we have, for any $k \ge 1$, $N \ge 1$,
  $$
\int_{\R} f_k(y)\,
 \big (\sum_{n=1}^N \nu_n \big) (dy)
 \le  \int_{\R}f_k(y) \tilde \nu (dy) \le
 \int_{\R} \big( y^2 \wedge 1 \big ) \tilde \nu (dy).
$$
Passing to the limit, as $k \to \infty$, in the left hand-side
 of the previous formula, we
get
 assertion (ii).

\hh (ii) $\Rightarrow $ (i). By using the inequality
$$
 \int_{\R} \big( 1 -  \cos  (h  y ) \big ) \nu_n (dy)
\le  \int_{\R} (1 \wedge  (hy)^2) \nu_n (dy), \;\; h \in \R, \;\;
n \ge 1,
 $$
we obtain  that  condition (ii) implies that  the series
 $$
\sum_{n=1}^{\infty}
 \int_{\R} \big( 1 -  \cos  (h  y ) \big ) \nu_n (dy)
$$
converges uniformly in $h$ on compact sets of $\R$. By the L\'evy
convergence theorem, this gives convergence in distribution
 of the sequence  $(\sum_{n = 1}^N \xi_n )$ and
 concludes the proof.
\epf

Applying the previous  result we can   clarify when our
cylindrical L\'evy
 process $Z = (Z_t)$ takes values in $H$.

\begin{proposition} \label{d} The following conditions are
equivalent.

$$ (i) \;\;\; \sum_{n \ge 1} (\beta_n Z_{t_0}^n)^2 < + \infty, \;\;\;\;
\P-a.s.,
 \;\; \text{for some  } \;\; t_0 >0;
$$
$$  (ii) \;\;\; \sum_{n \ge 1} (\beta_n Z_t^n)^2 < + \infty, \;\;\;\;
\P-a.s.,
 \;\; \text{for any } \;\; t \ge 0;
$$
$$ \;\;\;\; (iii) \;\;\; \sum_{n \ge 1} \Big ( \beta_n^2  \int_{| y| < 1 / \beta_n }
y^2  \nu (dy) \, +\, \int_{| y| \ge 1 / \beta_n }   \nu (dy) \Big)
< + \infty.
$$
\end{proposition}
\begin{proof}
We first show that (i) implies  (iii).  Assertion (i) is
equivalent to convergence in distribution
 of the sequence of random variables $(\sum_{n = 1}^N \beta_n Z_{t_0}^n)$
 (see the result mentioned at the beginning of the proof of Lemma
 \ref{indip}).
 We have, for any $n \ge 1$, $h \in \R$,
$$
  \E  [e^{i h  \beta_n Z_{t_0}^n } ] =  e^{- {t_0}
 \int_{\R} \big( 1 -  \cos  (h \beta_n y ) \big ) \nu (dy)}
 =    e^{- {t_0}
 \int_{\R} \big( 1 -  \cos  (h y ) \big )
   \nu_n  (dy)},
$$
where $\nu_n$ is the image law of $\nu$ by the transformation:
 $ y \mapsto \beta_n y$. Setting $\xi_n = \beta_n Z_{t_0}^n$,
 by Lemma \ref{indip}  assertion (i) is equivalent to
 $$
 \sum_{n \ge 1} \int_{\R} \big( 1 \wedge y^2 \big )
\nu_n (dy) < + \infty.
$$
Now assertion (iii)  follows since, for any $n \in \N,$
$$
\int_{\R} \big( y^2 \wedge 1 \big ) \nu_n  (dy) =
  \int_{|y| <  1}  y^2
  \nu_n  (dy) + \int_{|y| \ge 1}
 \nu_n  (dy)
$$
$$
=  \int_{|\beta_n y| <  1}  \beta_n^2 y^2
  \nu  (dy) + \int_{|\beta_n y| \ge 1}
 \nu  (dy).
$$
Using again Lemma \ref{indip} we get that (iii) implies (ii) as
well. The proof is complete.
\end{proof}

\begin{remark} \label{pz}
{\em Theorem 4.13 in  \cite{PeZ} states  that
 if condition (ii)  in Proposition \ref{d} holds then also (iii) is
 satisfied. However such theorem does not require symmetricity of
 the L\'evy process $Z$.}
 \end{remark}

\begin{remark} \label{by}
{\em If $(Z_t^n)$ are symmetric $\alpha$-stable processes,
 $\alpha \in (0,2) $,  then $\nu (dy) = \frac{1}{|y|^{1+ \alpha}} dy$
  and so  (ii) of Proposition \ref{d} is equivalent to
  $$
 \sum_{n \ge 1} \beta_n^{\alpha} < + \infty
 $$
as in \cite{PZ7}. }
 \end{remark}

\begin{remark} \label{levy} {\em
 Using
  Proposition \ref{d},
  one gets that our cylindrical L\'evy process $Z$
  is a L\'evy process with values
  in the space $l^2_{\rho}$, see (\ref{rho}),
   where $(\rho_n)$ is a sequence of positive numbers
    such that
 $$
 \sum_{n \ge 1}  \Big ( (\rho_n\beta_n)^2  \int_{|\rho_n \beta_n  y| < 1 }
y^2  \nu (dy) \, +\, \int_{|\rho_n \beta_n y| \ge 1 }   \nu (dy)
\Big) < + \infty.
 $$
}
\end{remark}
Let us come back to the Ornstein-Uhlenbeck process.
 According to Hypothesis \ref{basic},
 we may consider our equation \eqref{ab}
 as an infinite sequence of independent
  one dimensional stochastic equations, i.e.,
 \beq \label{eq}
 dX^n_t = -
 \gamma_n X^n_t dt + \beta_n d Z^n_t,\;\;\;\; X^n _0 = x_n, \;\;
 n \in \N,
\eeq
 with   $x = (x_n) \in l^2 = H$.  The  {\it solution} is
 a stochastic process $X = (X_t^x) $ which takes  values in
   $ \R^{\N}$ with  components
 \beq \label{xt}
 X^n_t  = e^{- \gamma_n t }x_n +
  \int_0^t e^{- \, \gamma_n (t- s)
 } \beta_n dZ_s^n,\;\;\; n \in \N,\;\; t \ge 0.
 \eeq

 \bth \label{bo1}
 Assume  Hypotheses \ref{bas} and \ref{basic} and consider the process
  $X = (X_t^x)$
   given in \eqref{xt}, $x \in H$. Define
     $\psi_{0} (u) = \int_{\{ | y| \le u \} } y^2 \nu (dy) $ and
 $\psi_{1} (u) = \int_{\{ |  y| > u \} }  \nu (dy) $.

   The following assertions are
  equivalent.
 \begin{align}\nonumber  & (i) \;\;\; X_{t_0}^x \in H, \;\;\;\; \P-a.s.,
    \;\; \text{for some  } \;\; {t_0} >0;
\\
\nonumber & (ii) \;\;\; X_t^x \in H, \;\;\;\; \P-a.s.,  \;\;
\text{for any } \;\; t\ge 0;
\\
\nonumber  & (iii) \;\;\; \sum_{n \ge 1} \frac{1}{\gamma_n}
 \int_{\frac{1}{\beta_n }}^ {\frac{1}{\beta_n } e^{\gamma_n  }}
  \Big ( \frac{1}{u^3} \psi_{0} (u) + \frac{1}{u} \psi_{1}
  (u) \Big )
  du < + \infty.
 \end{align}
 Moreover, under one of the previous assertions,  we have
  \begin{align} \label{mildo}  &  X_t^x = \sum_{n \ge 1} X_t^n e_n  =
e^{tA} x + Z_A(t),\;\;\; \mbox{where} \;\; \\
\nonumber &  Z_A(t) = \int_0^t e^{(t-s)A} dZ_s = \sum_{n \ge 1}
\Big( \int_0^t e^{- \, \gamma_n (t- s)
 } \beta_n dZ_s^n \Big) e_n,
  \end{align}
   and  the process
   $(X_t^x)$ is ${\cal F}_t$-adapted  and
    Markovian.
 \eth \bpf \noindent  {\bf I step. } We  show that (i) is equivalent to the following
 condition
\begin{equation} \label{ct}
\sum_{n \ge 1} \frac{1}{\gamma_n}
 \int_{\frac{1}{\beta_n }}^ {\frac{1}{\beta_n } e^{\gamma_n {t_0} }}
  \Big ( \frac{1}{u^3} \psi_{0} (u) + \frac{1}{u} \psi_{1}
  (u) \Big )
  du < + \infty.
\end{equation}
  Let  us consider
   the stochastic convolution
 \beq \label{yy}
Y_t^n=  Z_A^n(t) = \int_0^t e^{- \, \gamma_n (t- s)
 } \beta_n dZ_s^n,\;\;\; n \in \N, \;\; t \ge 0,
 \eeq
 where the stochastic integral is a limit
  in probability of Riemann sums.
 We have, for any $h \in \R,$ see \eqref{ps},
 \beq \label{ritorna}
    \E  [e^{i h Y^n_{t_0}} ]
   = \exp \Big[
   -  \int_0^{t_0} \psi \big ( e^{-
 \gamma_n \, s } \beta_n h \big) ds \Big]
\eeq
 where $\psi $ is given in \eqref{ps}.
 By the Fubini theorem
 $$
\int_0^{t_0} \psi \big ( e^{-
 \gamma_n \, s } \beta_n h \big) ds =
\int_0^{t_0} ds
 \int_{\R} \big( 1 -  \cos  ( e^{- \gamma_n s } \beta_n
  h  y )  \big) \nu (dy)
$$
$$
= \int_0^{t_0} ds
 \int_{\R} \big( 1 -  \cos  ( hy
   )  \big) \nu_{ns} (dy) =
  \int_{\R} \big( 1 -  \cos  (h  y ) \big )
  \tilde \nu_{n}  (dy),
$$
 where $\nu_{ns}$ is the image law of
  $\nu$ by the transformation:
  $ y \mapsto \beta_n e^{- \gamma_n s } y$ and
   we have set
\begin{equation}\label{e2} \tilde \nu_{n}  (B)=\big(
  \int_0^{t_0} \nu_{ns} ds\big) (B) =
  \int_0^{t_0}  \Big (
  \int_{\R} I_{B} (\beta_n e^{- \gamma_n s } y ) \nu(dy)
  \Big) ds,
\end{equation}
for any Borel set $B \subset \R$ ($I_B$ is the indicator function
of $B$).
Setting $\xi_n =  Y_{t_0}^n$,
 by Lemma \ref{indip},  assertion (i) is equivalent to
  \begin{equation} \label{d4}
 \sum_{n \ge 1} \int_{\R} \big( 1 \wedge y^2 \big )
\tilde \nu_n (dy) =
 \sum_{n \ge 1} \int_0^{t_0} \Big( \int_{\R} \big( 1 \wedge y^2 \big )
 \nu_{ns} (dy) \Big) ds < + \infty.
\end{equation}
 Let us fix $n \ge 1$. We have
  $$
 \int_0^{t_0} \Big( \int_{\R} \big( 1 \wedge y^2 \big )
 \nu_{ns} (dy) \Big) ds =
  \int_0^{t_0} \Big( \int_{\R} \big( 1 \wedge (\beta_n e^{- \gamma_n s } y
  )^2 \big )
 \nu_{} (dy) \Big) ds
$$
$$
 = \int_0^{t_0}
  \Big( \beta_n^2
  e^{- 2\gamma_n s }  \int_{|y| \le \frac{e^{ \gamma_n s }}{\beta_n}}
  y^2
 \nu_{} (dy) \, +\,  \int_{|y| > \frac{{e^{ \gamma_n s }}}{\beta_n}}
 \nu_{} (dy) \Big) ds
$$
$$
=  \frac{1}{\gamma_n } \int_{\frac{1}{\beta_n}}^{\frac{e^{
\gamma_n t_0 }}{\beta_n} }
  \Big( \frac{1}{u^3}\int_{|y| \le u}
  y^2
 \nu_{} (dy) \, +\,  \frac{1}{u} \int_{|y| > u }
 \nu_{} (dy) \Big) du.
$$
This shows that \eqref{d4} is exactly \eqref{ct}.

\smallskip
\noindent  {\bf II step. } In order to prove equivalence between
(i), (ii) and
  (iii)  it remains to  show that
 (i) implies (ii).

 Note that if \eqref{ct} holds for ${t_0}>0$, then it is also
satisfied for any $0 \le s \le {t_0}$. Therefore, assertion (i)
implies that
 \begin{equation} \label{45}
 X_s^x \in H, \;\;\;\; \P-a.s.,
    \;\; \text{for any  } \;\; s \in [0,{t_0}].
\end {equation}
and so $Z_A(s) \in H, \; \P-a.s.,
    \; \text{for any  } \; s \in [0,{t_0}].$

We have  the following identity on the product space $\R^{\N}$,
$\P$-a.s.,
 \begin{equation} \label{mar}
 Z_A(T+h ) - e^{hA} Z_A (T) = \int_T^{T+h} e^{(T+h
 -s)A}dZ_s = \int_0^{h} e^{(h
 - u)A}dZ_u^T,
\end{equation}
for any $T, h \ge 0$, where $Z_u^T = Z_{T+u} - Z_T$, $u \ge 0$, is
 still a L\'evy process with values in $\R^{\N}$.
  Note that
  $$\int_0^{h} e^{(h
 - u)A}dZ_u^T $$
 has the same law of $Z_A(h)$.

Combining \eqref{45} and identity \eqref{mar} with $T={t_0}$ and
$h \in [0,{t_0}]$, we deduce that $\P (Z_A (r) \in H ) =1$, for
any $r \in [{t_0},2{t_0}]$. By an iteration procedure, we infer
that $\P (Z_A (r) \in H ) =1$, for any $r \ge 0$.  This
immediately implies condition (ii).
 The first part of the proof is finished.


\smallskip
\noindent  {\bf III step. }  The property  that $(X^x_t)$ is
${\cal F}_t$-adapted is
  equivalent to
 the fact that each
  real process $\langle X_t^x, e_k \rangle$ is ${\cal F}_t$-adapted,
   for any $k \ge  1$, and this  clearly holds.

 The Markov property follows easily from the identity
  \eqref{mar}.
   \epf

\begin{remark}
{\em  Note that
$$
\lim_{\gamma \to 0} \frac{1}{\gamma}
 \int_{\frac{1}{\beta_n }}^ {\frac{1}{\beta_n } e^{\gamma  }}
  \Big ( \frac{1}{u^3} \psi_{0} (u) + \frac{1}{u} \psi_{1}
  (u) \Big )du = \beta_n ^2 \psi_0 (\frac{1}{\beta_n } )
  + \psi_1(\frac{1}{\beta_n } ).
$$
This shows that Proposition \ref{d} is a ``limiting case''
 of Theorem \ref{bo1},
obtained  when $\gamma_n =0$, for any $n \ge 1$. }\end{remark}

\begin{remark}\label{kotelenez}
{\em If the cylindrical L\'evy process $Z$  takes values in the
Hilbert space $H$, i.e.,
 if  condition (ii) of Proposition \ref{d} holds,  then, by the Kotelenez
regularity result
 (see \cite[Theorem 9.20]{PeZ})  trajectories of the process $X$
 which solves \eqref{ab} are c\`adl\`ag with values in $H$.
  However such condition (ii)
  is a very restrictive assumption (see also
  Remark \ref{by}). We conjecture that the
c\`adl\`ag property holds under much weaker conditions but, at the
moment, this is an open problem. }
\end{remark}

In the next result we provide an application of  Theorem \ref{bo1}
 to  Ornstein-Uhlenbeck processes driven by a quite general class of
  symmetric  cylindrical L\'evy noises (this class  in particular
   includes the $\alpha$-stable cylindrical processes).

\bpr \label{bo11}
 Assume  Hypotheses \ref{bas} and \ref{basic}.  Moreover, assume that
  $(\beta_n)$ is a bounded sequence and that the symmetric L\'evy
   measure $\nu$ appearing in \eqref{psu} satisfies
 \beq
\label{psu1}
  \int_{1}^{+ \infty} \log(y) \,  \nu (dy) < + \infty.
\eeq
Finally, assume that
$$
\sum_{n \ge 1}\frac{1}{\gamma_n} < + \infty.
$$
 Then  the Ornstein-Uhlenbeck  process
  $X = (X_t^x)$
   given in \eqref{xt}  verifies  assertions (i)-(iii) of Theorem \ref{bo1}.
   Moreover, $X$ has an unique invariant measure.
    \end{proposition}
 \bpf
  First remark that (ii) of Theorem \ref{bo1} is equivalent to
  $$
\sum_{n \ge 1} \beta_n^2
\Big( \int_0^t e^{- \, \gamma_n (t- s)
 }  dZ_s^n \Big)^2  < +\infty,\;\;\; \P-a.s., \;\;\; t \ge 0.
$$
Therefore, it is enough to check the result assuming that $\beta_n=1$,
for any $n \ge 1$.

 We will check condition (iii), i.e.,
  \beq \label{fr}
 \sum_{n \ge 1} \frac{1}{\gamma_n}
 \int_{{1}}^ { e^{\gamma_n  }}
  \Big ( \frac{1}{u^3} \psi_{0} (u) + \frac{1}{u} \psi_{1}
  (u) \Big )
  du < + \infty.
\eeq
Let us fix $0< 1< b$. We first estimate the function $f_0$,
$$
 f_0 (b) = \int_{{1}}^ { b }
  \Big ( \frac{1}{u^3} \psi_{0} (u) + \frac{1}{u} \psi_{1} (u) \Big )
  du.
$$
We have, by using symmetricity and Fubini theorem,
$$
\int_1^b \frac{1}{u} \psi_{1}
  (u)  du  =  2  \int_1^b \frac{1}{u} \big( \int_{y >u}
  \nu (dy) \big) du
$$
$$
=  2  \int_1^b \frac{1}{u} du  \int_0^{+\infty}
  1_{(u, +\infty)} (y)
  \nu (dy) =   2  \int_0^{+\infty} \Big( \int_1^b
    \frac{1}{u}
  1_{(u, +\infty)} (y) du \Big)
  \nu (dy)
$$
$$
=  2  \int_0^{+\infty} \Big( \int_{1 \wedge y}^{b \wedge y}
    \frac{du}{u}  \Big)
  \nu (dy) =  2  \int_0^{+\infty} \big( \log (b \wedge y)
   -  \log (1 \wedge y) \big)
  \nu (dy)
$$
$$
  = 2  \int_1^{b}  \log ( y)
  \nu (dy) +
     2  \log (b) \, \nu ((b,{+\infty}) )
$$
and, similarly,
$$
\int_1^b \frac{1}{u^3} \psi_{0}
  (u)  du  =  2  \int_1^b \frac{1}{u^3}   du  \int_0^{+\infty}
  1_{[0, u]} (y)
  \, y^2\, \nu (dy)
  $$
$$
=  2  \int_0^{+\infty} y^2  \Big( \int_{1 \vee y}^{b \vee y}
    \frac{du}{u^3}  \Big)
  \nu (dy) =    \int_0^{+\infty} y^2 \, \Big( \frac{1}{(1 \vee y)^2}
   -  \frac{1}{(b \vee y)^2} \Big)
  \nu (dy)
$$
$$
= ( {1}
   -  \frac{1}{b ^2})  \int_0^{1} y^2
  \nu (dy) +
    \int_1^{b} y^2 \, \Big( \frac{1}{ y^2}
   -  \frac{1}{b^2} \Big)
  \nu (dy)
$$
$$
  =   \int_1^{b}
  \nu (dy) +    \int_0^{1} y^2
  \nu (dy) - \frac{1}{b^2} \int_0^{b} y^2
  \nu (dy).
$$
We have the following estimate, for any $b \in (1, \infty)$,
$$
0 \le  f_0 (b) \le   2  \int_1^{+ \infty}  \log ( y)
  \nu (dy) +
     2  \log (b) \, \nu ((b,{+\infty}) )
  +
 \int_1^{+ \infty }
  \nu (dy) +    \int_0^{1} y^2
  \nu (dy).
$$
 Setting $C =  2  \int_1^{+ \infty}  \log ( y)
  \nu (dy) +
 \int_1^{+ \infty }
  \nu (dy) +
  \int_0^{1} y^2
  \nu (dy) < + \infty $, we find
$$
  f_0 (b) \le C +  2  \frac{\log (b)}{\log (b)}
  \int_b^{+ \infty}  \log ( y)
  \nu (dy) \le 3 C, \;\;\; b \ge 1.
$$
Note that \eqref{fr} is equivalent to
 $$
 \sum_{n \ge 1} \frac{1}{\gamma_n}
  f_0{ (e^{\gamma_n  })} \le 3 C
   \sum_{n \ge 1} \frac{1}{\gamma_n}
   < + \infty.
$$
The proof of the first part of the theorem is complete.

 To show that there exists an invariant measure we first note that
(according to \cite[Theorem 17.5]{sato}) each
 one dimensional Ornstein-Uhlenbeck process  $(X_t^n)$ has an
  invariant
 measure $\mu_n$  which is the law of the random variable
 $$
 \int_0^{\infty} e^{- \, \gamma_n u
 } \beta_n dZ_u^n
$$
 having characteristic function $\hat \mu_n (h)
= \exp \Big( - \int_0^{\infty} \psi (e^{- \, \gamma_n
 s} \beta_n h) ds \Big)$, $h \in\R$.

Let us consider the product measure $\mu = \prod_{n \ge 1} \mu_n$
 on $\R^{\N}$. This is the law of the $\R^{\N}$-random variable
  $\xi = (\xi_n)$, where
  $$
 \xi_n = \int_0^{\infty} e^{- \, \gamma_n u
 } \beta_n dZ_u^n, \;\;\; n \ge 1.
$$
 According to Lemma \ref{indip},
 $\xi$ takes values in $H$ if and only if the L\'evy measures
  $\nu_n$ of $\xi_n$ verify
  $$
 \sum_{n \ge 1} \int_{\R} \big( 1 \wedge
y^2 \big ) \nu_n (dy) < + \infty.
$$
This condition is equivalent to
$$
\sum_{n \ge 1} \frac{1}{\gamma_n}
 \int_{1}^ { +\infty }
  \Big ( \frac{1}{u^3} \psi_{0} (u) + \frac{1}{u} \psi_{1}
  (u) \Big )
  du  = \sum_{n \ge 1} \frac{1}{\gamma_n}
  \, \Big(\sup_{b \ge 1} f_0 (b) \Big)
   < + \infty
$$
 which holds. This shows that $\mu (H) =1$ and so $\mu$
  in a Borel probability measure on $H$.

We will prove that $\mu$ is the unique invariant measure of
 $X$ by showing that, for any $x \in H$,
 \begin{equation} \label{inv}
 \lim_{t \to \infty} X^x_t = \xi
\end{equation}
 in probability  (see \cite{DZ1}). It is enough to prove \eqref{inv}
 when $x=0$. Let $X_t^0 = Y_t$ and fix any $\epsilon >0$.
  By using  characteristic function, one checks easily that
    the law of $Y_t$ is the same as the one of
 $
  \int_0^{t} e^{- \, \gamma_n u
 } \beta_n dZ_u^n.
 $
We find
$$
 a_t = \P ( |Y_t - \xi |^2 > \epsilon) =
 \P \big( \sum_{n \ge 1}
  \beta_n^2 \, \big(\int_t^{\infty} e^{- \, \gamma_n u
 }  dZ_u^n   \big)^2 > \epsilon \big)
 $$
Now, for any $t>0$, we consider new independent L\'evy processes
 $(Z_r^{t,n})_{r \ge 0}$, where
 $Z_r^{t,n} = Z_{r+t}^{n} - Z_{t}^{n} $, $r \ge 0$, $n \ge 1$.
    For any $t>0$, $\int_t^{\infty} e^{- \, \gamma_n u
 }  dZ_u^n$ has the same law as
 $$
  \int_0^{\infty} e^{- \, \gamma_n (t+ s)
 }  d Z_s^{t,n} =  e^{- \, \gamma_n t
 } \int_0^{\infty} e^{- \, \gamma_n s
 }  d Z_s^{t,n}
$$
 which coincides with the law of $e^{- \, \gamma_n t
 } \frac{\xi_n}{\beta_n}$. By using independence, for any
 $t>0$, the law of
   $ \sum_{n \ge 1} \beta_n^2
   \big(\int_t^{\infty} e^{- \, \gamma_n u
 }  dZ_u^n   \big)^2
$ coincides with the one of
$$
\sum_{n \ge 1} e^{- \, 2\gamma_n t} \xi_n^2.
$$
Assume that $\gamma_n \ge \gamma_0 >0$, $n \ge 1$. We have, for any $t>0$,
$$
 a_t =  \P \big( \sum_{n \ge 1}
   \, e^{- \, 2\gamma_n t} \xi_n^2
    > \epsilon \big) \le
     \P \big( e^{- \, 2\gamma_0 t} \, \sum_{n \ge 1}
   \, \xi_n^2
    > \epsilon \big) =  \P \big(  | \xi|^2
    \, > e^{\, 2\gamma_0 t} \epsilon \big).
$$
 By Letting $t \to \infty$, we find  $\lim_{t \to \infty} a_t =0$.
 This proves \eqref{inv} with $x=0$ and concludes the proof.
\end{proof}

\begin{example}\label{E1}   {\em Consider
  the  following linear stochastic  heat
  equation  on $D= [0, \pi]^d$ with
  Dirichlet boundary conditions
 \beq \label{heat1}
  \left\{
\bal dX(t, \xi)  &  = \triangle X(t, \xi) \,dt  + \,dZ(t, \xi), \;\;\; t>0,\\
      X(0, \xi) & = x(\xi), \;\;\; \xi \in D,
\\ X(t, \xi) &=0, \;\; t>0, \;\;\; \xi \in \partial D,
\eal \right. \eeq
 where $Z$ is a  cylindrical L\'evy process with
respect to the basis of  eigenfunctions of the Laplacian $\Delta$
in $H = L^2(D)$ (with  Dirichlet boundary conditions). The
eigenfunctions are
 $$
 e_j (\xi_1, \ldots, \xi_d) = (\sqrt{2/\pi})^d \sin (n_1 \xi_1)
 \cdots \sin(n_d \xi_d), \;\;\; \xi = (\xi_1, \ldots, \xi_d) \in
 \R^d,  $$
 $j=(n_1, \ldots, n_d) \in
 \N^d$.
 The corresponding  eigenvalues
 are $-\gamma_j$, where $\gamma_j =  (n_1^2 + \ldots + n_d^2 )$.
 The operator $A = \triangle $ with
  $D(A) = H^2(D) \cap H^1_0(D)$ verifies
   Hypothesis \ref{basic}.
} \qed
\end{example}

\section{Irreducibility  }

We start with a simple lemma, which we prove for the reader
  convenience.

\ble \label{facile}
 Let us consider a sequence $(\xi_n)$ of
  independent real random variables,  defined on the same
   probability space
$(\Omega, {\cal F}, \P)$ such that
 $$
 \sum_{n \ge 1} \xi_n^2 < + \infty,\;\;\; \P-a.s..
$$
If each $\xi_n $ has full support in $\R$, then the random
variable $ \xi =  (\xi_n) $ has full support in
 the Hilbert space $l^2$.
\ele
 \bpf
 We fix an arbitrary ball $B \subset l^2$, $B = B(y, r)$ with center
 in $y= (y_k) \in l^2$ and radius $r>0$.
 Using independence, we find
 \beqr
&& \P \Big(\sum_{k \ge 1} (\xi_k - y_k)^2 < r^2 \Big)
 \ge \P \Big(\sum_{k = 1}^N (\xi_k - y_k)^2 < \epsilon,
 \, \sum_{k > N} ( \xi_k - y_k)^2 < r^2 - \epsilon \Big)
\\ && \ge \P \Big(\sum_{k = 1}^N ( \xi_k - y_k)^2 < \epsilon \Big)
 \, \P \Big(\sum_{k > N} ( \xi_k - y_k)^2 < r^2 - \epsilon \Big).
   \eeqr
Now we use that each $\xi_n$ has full support in  $\R$. This
implies that, for any $N \in \N$, $\epsilon
>0$,
 $ \P \Big(\sum_{k = 1}^N ( \xi_k - y_k)^2 < \epsilon \Big) >0.
 $ Since $\P (
 \, \sum_{k > N} (\xi_k - y_k)^2 < r^2 - \epsilon) \to 1$,
 as $N \to \infty$, the assertion follows.
\epf

We  need the following result, which is a consequence   of
\cite[Theorem 24.10]{sato}.

\bth \label{Sato} Consider
  a  symmetric  infinitely divisible law $\mu$ on $\R$.
   If the support of its  L\'evy measure
 $M$  contains $0$
 (i.e.,   for any $\delta>0$,  $M(
 (- \delta, \delta) ) >0\, )$,
  than the support
 of $ \mu $ is $\R$.
  \eth
 \bpf  Arguing as in the proof of Proposition
  \ref{use} we get that  $M$ is symmetric.
 Therefore, the support of $M $, which contains $0$,
    has non-empty intersection
  with $(0, + \infty)$ and with $(- \infty, 0)$. By assertion
 (ii) in \cite[Theorem 24.10]{sato}, we get that the support of
  $\mu $ is $\R$.
\epf

Now  we  prove irreducibility of solutions
 to \eqref{ab}.

  \bth \label{irr}   Assume  Hypotheses \ref{bas} and
   \ref{basic}. Moreover,
  suppose  that  the support of the L\'evy measure $\nu$
    given in \eqref{psu}   contains $0$.

     Then, for
any $x \in H$, the OU
 process $(X_t^x)$ given in \eqref{xt}
  is irreducible, that is,  for any open ball $B \subset H$, $t>0,$
 we have $ \P (X^x_t \in B)
>0.$
  \eth

\bpf According to Lemma \ref{facile}, it is enough to prove that
 the one dimensional Ornstein-Uhlenbeck processes, starting from 0,
 $$
 Y_t^n=  Z_A^n(t) = \int_0^t e^{- \, \gamma_n (t- s)
 } \beta_n dZ_s^n,\;\;\; n \in \N, \;\; t > 0,
$$
are irreducible. To this purpose, we fix $t>0$ and denote by
 $\mu_n$  the symmetric and infinitely divisible law
   of $Y_t^n$,
 having L\'evy measure $\tilde \nu_{n}$ of the form (\ref{e2}).

Using the fact that   $\mu_n$ is symmetric and  Theorem
\ref{Sato},
  to show  that the   support of $\mu_n$ is full in $\R$,
   we need to  check that the support of $\tilde \nu_{n}$
 contains 0. This follows easily from
 the assumption on $\nu$ and formula \eqref{e2}.
 \epf

\end{document}